\documentclass{article}
\usepackage{amsmath,amssymb,amsthm}
\newtheorem{theorem}{Theorem}
\newtheorem{lemma}{Lemma}
\begin{document}
\title{Decay of scattering solutions to one-dimensional free Schr\"{o}dinger equation}
\author{Yuya Dan\\Matsuyama University\\dan@cc.matsuyama-u.ac.jp}
\maketitle

\abstract{
In this paper, we investigate the decay property of scattering solutions to the initial value problem for the free Schr\"{o}dinger equation in $\mathbb{R}$.
It becomes clear that the rate of time decay is essentially determined by the behavior of the Fourier transform of initial data near the origin. The proof is described by basic calculus.
}

\section{Introduction}
Let us consider the free Schr\"{o}dinger equation
\begin{equation}
	i \frac{\partial}{\partial t} u( t, x ) = - \frac{1}{2} \frac{\partial^{2}}{\partial x^{2}} u( t, x )
	\label{Schroedinger}
\end{equation}
with initial data $u( 0, x ) = u_{0}( x )$, where $i = \sqrt{-1}$ and $u$ is a complex-valued function of $t \in \mathbb{R}$ and $x \in \mathbb{R}$. If $u_{0}$ is a squared-integrable function, then we can solve the equation (\ref{Schroedinger}),
such as
\begin{equation}
	u( t, x )
	= ( 2 \pi )^{- \frac{1}{2}} \int_{-\infty}^{\infty} e^{i x \xi} e^{- \frac{1}{2} i t \xi^{2}} \hat{u}_{0}( \xi ) \ d \xi,
\end{equation}
where $\hat{u}_{0}$ stands for the Fourier transform of $u_{0}$. For any squared-integrable function $v$ in $\mathbb{R}$, we define
\begin{equation}
	\hat{v}( \xi ) = ( 2 \pi )^{-\frac{1}{2}} \int_{-\infty}^{\infty} e^{-i x \xi} v( x ) \ dx.
\end{equation}

There are two types of solutions to the free Schr\"{o}dinger equation classifies by the behavior as $t \to \infty$.
 Let $\lambda$ be an eigenvalue of the operator $-\frac{1}{2} \frac{\partial^{2}}{\partial x^{2}}$ and $\psi$ be the corresponding eigenfunction, that is,
\begin{equation}
	-\frac{1}{2} \frac{\partial^{2}}{\partial x^{2}} \psi( x ) = \lambda \psi( x ).
\end{equation}
Then, $u( t, x ) = e^{-i \lambda t} \psi( x )$ is a solution to the equation although $u$ may not belong as a function of $x$, to the class of squared-integrable functions. This $u( t, x )$  does not decay in time as $\lvert t \rvert \to \infty$, because the absolute value of $e^{-i \lambda t}$ is one for any $t$ in $\mathbb{R}$. A solution of this type is called a bound-state solution.

On the other hand, a solution which decays in time as $\lvert t \rvert \to \infty$ is called a scattering solution. As an example of scattering solutions, we give $u_{0}( x ) = e^{-\frac{1}{2} x^{2}}$, then we can explicitly solve the equation,
\begin{equation}
	u( t, x ) = \{ 2 \pi ( 1 + i t ) \}^{-\frac{1}{2}} e^{-\frac{x^{2}}{2 ( 1 + i t )}},
\end{equation}
which decays at the rate of $\lvert t \rvert^{-\frac{1}{2}}$ as $\lvert t \rvert \to \infty$.
In this paper, we investigate the decay property of scattering solution $u( t, x )$ as $\lvert t \rvert \to \infty$.

Before describing the main result, we define function spaces. Let $\Omega$ be an open set in $\mathbb{R}$.
We say $v \in L^{2}( \Omega )$ if the quantity
\begin{equation}
	\int_{\Omega} \lvert v( x ) \rvert^{2} \ dx
	\label{quantity}
\end{equation}
is finite. $L^{2}( \Omega )$ is a Hilbert space with the inner product
\begin{equation}
	( u, v ) = \int_{\Omega} u( x ) \overline{v( x )} \ dx
\end{equation}
for any $u \in L^{2}( \Omega )$ and $v \in L^{2}( \Omega )$. We use the norm $\lVert u \Vert_{L^{2}( \Omega )} = \sqrt{( u, u )}$ throughout the paper. $\lVert u \Vert_{L^{2}( \Omega )}$ is equal to the square root of (\ref{quantity}).

Then, we can prove the following result.

\begin{theorem}
Let $x u_{0} \in L^{2}( \mathbb{R} )$.
If
\begin{equation}
	\xi^{-2} \hat{u}_{0}( \xi ) \in L^{2}( \lvert \xi \rvert < 1 )
\end{equation}
and
\begin{equation}
	\xi^{-1} \left( \frac{\partial}{\partial \xi} \hat{u}_{0} \right)( \xi ) \in L^{2}( \lvert \xi \rvert < 1 ),
\end{equation}
then
\begin{equation}
	\lvert u( t, x ) \rvert
	\leq 
	C ( 1 + \lvert x \rvert ) \lvert t \rvert^{-1},
\end{equation}
for any $t \in \mathbb{R}$ and $x \in \mathbb{R}$, where $C > 0$ is a constant which is independent both of $t$ and $x$.
\end{theorem}

It should be remarked that the region $\lvert \xi \rvert < 1$ is not essential.
This may be changed into any compact set which contains $\xi = 0$.

\section{Lemmas}

In this section, we discuss some lemmas in order to prove the main theorem later.

\begin{lemma}
Let $u_{0} \in L^{2}( \mathbb{R} )$, $\xi^{-1} \hat{u}_{0} \in L^{2}( \lvert \xi \rvert < 1 )$ and
\begin{equation}
	I_{1}( t, x ) = \int_{-\infty}^{\infty} e^{i x \xi} e^{-\frac{1}{2} i t \xi^{2}} \xi^{-1} \hat{u}_{0}( \xi ) \ d\xi.
\end{equation}
Then
\begin{equation}
	\lvert I_{1}( t, x ) \rvert \leq \sqrt{2} \lVert \xi^{-1} \hat{u}_{0}( \xi ) \rVert_{L^{2}( \lvert \xi \rvert < 1 )} + \sqrt{2} \lVert u_{0} \rVert_{L^{2}( \mathbb{R} )}.
\end{equation}
\end{lemma}

\begin{proof}
We have
\begin{equation}
	\begin{array}{ccl}
	\lvert I_{1}( t, x ) \rvert
	&\leq& \displaystyle\int_{-\infty}^{\infty} \lvert \xi^{-1} \hat{u}_{0}( \xi ) \rvert \ d\xi \\
	&=& \displaystyle\int_{\lvert \xi \rvert < 1} \lvert \xi^{-1} \hat{u}_{0}( \xi ) \rvert \ d\xi + \int_{\lvert \xi \rvert \geq 1} \lvert \xi^{-1} \hat{u}_{0}( \xi ) \rvert \ d\xi. \\
	\end{array}
\end{equation}
By Schwarz' inequality, we obtain
\begin{equation}
	\begin{array}{ccl}
		\displaystyle
		\int_{\lvert \xi \rvert < 1} 1 \cdot \lvert \xi^{-1} \hat{u}_{0}( \xi ) \rvert \ d\xi
		& \leq & \displaystyle \left( \int_{\lvert \xi \rvert < 1} 1^{2} \ d\xi \right)^{\frac{1}{2}}
		\left( \int_{\lvert \xi \rvert < 1} \lvert \xi^{-1} \hat{u}_{0}( \xi ) \rvert^{2} \ d\xi \right)^{\frac{1}{2}} \\
		& = & \sqrt{2} \lVert \xi^{-1} \hat{u}_{0}( \xi ) \rVert_{L^{2}( \lvert \xi \rvert < 1 )}.
	\end{array}
\end{equation}
On the other hand, we have similarly
\begin{equation}
	\int_{\lvert \xi \rvert \geq 1} \lvert \xi^{-1} \hat{u}_{0}( \xi ) \rvert \ d\xi
	\leq
	\left( \int_{\lvert \xi \rvert \geq 1} \lvert \xi \rvert^{-2} \ d\xi \right)^{\frac{1}{2}}
	\left( \int_{\lvert \xi \rvert \geq 1} \lvert \hat{u}_{0}( \xi ) \rvert^{2} \ d\xi \right)^{\frac{1}{2}}.
\end{equation}
Plancerel theorem implies
\begin{equation}
	\int_{\lvert \xi \rvert \geq 1} \lvert \hat{u}_{0}( \xi ) \rvert^{2} \ d\xi
	\leq
	\lVert \hat{u}_{0} \rVert_{L^{2}( \mathbb{R} )}^{2}
	=
	\lVert u_{0} \rVert_{L^{2}( \mathbb{R} )}^{2},
\end{equation}
which concludes our desired estimate.
\end{proof}

\begin{lemma}
Let $u_{0} \in L^{2}( \mathbb{R} )$, $\xi^{-2} \hat{u}_{0} \in L^{2}( \lvert \xi \rvert < 1 )$ and
\begin{equation}
	I_{2}( t, x ) = \int_{-\infty}^{\infty} e^{i x \xi} e^{-\frac{1}{2} i t \xi^{2}} \xi^{-2} \hat{u}_{0}( \xi ) \ d\xi.
\end{equation}
Then
\begin{equation}
	\lvert I_{2}( t, x ) \rvert \leq \sqrt{2} \lVert \xi^{-2} \hat{u}_{0}( \xi ) \rVert_{L^{2}( \lvert \xi \rvert < 1 )} + C_{2}
\end{equation}
for some $C_{2} > 0$, independent both of $t$ and $x$.
\end{lemma}

\begin{proof}
By similar discussion to the previous lemma, we obtain
\begin{equation}
	\begin{array}{ccl}
	\lvert I_{2}( t, x )
	&\rvert& \leq \displaystyle\int_{-\infty}^{\infty} \lvert \xi^{-2} \hat{u}_{0}( \xi ) \rvert \ d\xi \\
	&\leq& \sqrt{2} \lVert \xi^{-2} \hat{u}_{0}( \xi ) \rVert_{L^{2}( \lvert \xi \rvert < 1 )} + C_{2} \\
	\end{array}
\end{equation}
where
\begin{equation}
	C_{2} = \int_{\lvert \xi \rvert \geq 1} \lvert \xi^{-2} \hat{u}_{0}( \xi ) \rvert \ d\xi,
\end{equation}
which completes the proof of the lemma.
\end{proof}

In the following lemma, $\partial_{\xi}$ denotes $\frac{\partial}{\partial \xi}$.
\begin{lemma}
Let $x u_{0} \in L^{2}( \mathbb{R} )$, $\xi^{-1} \partial_{\xi} \hat{u}_{0} \in L^{2}( \lvert \xi \rvert < 1 )$ and
\begin{equation}
	I_{3}( t, x ) = \int_{-\infty}^{\infty} e^{i x \xi} e^{-\frac{1}{2} i t \xi^{2}} \xi^{-1} ( \partial_{\xi} \hat{u}_{0} )( \xi ) \ d\xi.
\end{equation}
Then
\begin{equation}
	\lvert I_{3} \rvert \leq \sqrt{2} \lVert \xi^{-1} ( \partial_{\xi} \hat{u}_{0} )( \xi ) \rVert_{L^{2}( \lvert \xi \rvert < 1 )} + C_{3}
\end{equation}
for some $C_{3} > 0$, independent both of $t$ and $x$.
\end{lemma}

\begin{proof}
Similary, we obtain
\begin{equation}
	\lvert I_{3}( t, x )
	\rvert \leq \int_{\lvert \xi \rvert < 1} \lvert \xi^{-1} ( \partial_{\xi} \hat{u}_{0} )( \xi ) \rvert \ d\xi
	+ \int_{\lvert \xi \rvert \geq 1} \lvert \xi^{-1} ( \partial_{\xi} \hat{u}_{0} )( \xi ) \rvert \ d\xi.
\end{equation}
By the theory of Fourier transformation, $x u_{0} \in L^{2}( \mathbb{R} )$ implies that there exists $\partial_{\xi} \hat{u}_{0}$ in $L^{2}( \mathbb{R} )$, so that
\begin{equation}
	\lvert I_{3}( t, x ) \rvert \leq \sqrt{2} \lVert \xi^{-1} ( \partial_{\xi} \hat{u}_{0} )( \xi ) \rVert_{L^{2}( \lvert \xi \rvert < 1 )} + C_{3},
\end{equation}
where
\begin{equation}
	C_{3} = \int_{\lvert \xi \rvert \geq 1} \lvert \xi^{-1} ( \partial_{\xi} \hat{u}_{0} )( \xi ) \rvert \ d\xi
\end{equation}
is a positive constant which is independent both of $t$ and $x$.
\end{proof}

\section{Proof of the main theorem}
We can prove the main theorem by basic technic in calculus. The key idea of the proof is
\begin{equation}
	\frac{\partial}{\partial \xi} e^{- \frac{1}{2} i t \xi^{2}}
	= ( - i t \xi ) e^{- \frac{1}{2} i t \xi^{2}},
\end{equation}
which implies
\begin{equation}
	u( t, x )
	= ( 2 \pi )^{- \frac{1}{2}} ( - i t )^{-1} \int_{-\infty}^{\infty} e^{i x \xi} \xi^{-1} \frac{\partial}{\partial \xi} \left[ e^{- \frac{1}{2} i t \xi^{2}} \right] \hat{u}_{0}( \xi ) \ d \xi.
\end{equation}
Integrating the above integral by part,
\begin{equation}
	\int_{-\infty}^{\infty} e^{i x \xi} \xi^{-1} \frac{\partial}{\partial \xi} \left[ e^{- \frac{1}{2} i t \xi^{2}} \right] \hat{u}_{0}( \xi ) \ d \xi
	=
	- \int_{-\infty}^{\infty}  \frac{\partial}{\partial \xi} \left[ e^{i x \xi} \xi^{-1} \hat{u}_{0}( \xi ) \right] e^{- \frac{1}{2} i t \xi^{2}} \ d \xi,
\end{equation}
if $u_{0} \in L^{2}( \mathbb{R} )$. It follows from Leibniz' rule that
\begin{equation}
	u( t, x )
	= ( 2 \pi )^{-\frac{1}{2}} t^{-1} ( - x I_{1}( t, x ) - i I_{2}( t, x ) + i I_{3}( t, x ) )
\end{equation}
where
\begin{equation}
	\begin{array}{l}
		\displaystyle
		I_{1}( t, x )
		= \int_{-\infty}^{\infty} e^{i x \xi} e^{-\frac{1}{2} i t \xi^{2}} \xi^{-1} \hat{u}_{0}( \xi ) \ d\xi \\
		\displaystyle
		I_{2}( t, x )
		= \int_{-\infty}^{\infty} e^{i x \xi} e^{-\frac{1}{2} i t \xi^{2}} \xi^{-2} \hat{u}_{0}( \xi ) \ d\xi \\
		\displaystyle
		I_{3}( t, x )
		= \int_{-\infty}^{\infty} e^{i x \xi} e^{-\frac{1}{2} i t \xi^{2}} \xi^{-1} ( \partial_{\xi} \hat{u}_{0} )( \xi ) \ d\xi. \\
	\end{array}
\end{equation}
From the results discussed in the previous section, we obtain
\begin{equation}
	\begin{array}{rl}
	\lvert u( t, x ) \rvert
	\leq 
	\displaystyle
	\frac{( 2 \pi )^{-\frac{1}{2}}}{\lvert t \rvert}
	\{ &
	( \sqrt{2} \lVert \xi^{-1} \hat{u}_{0}( \xi ) \rVert_{L^{2}( \lvert \xi \rvert < 1 )} + \sqrt{2} \lVert u_{0} \rVert_{L^{2}( \mathbb{R} )} ) \lvert x \rvert \\
	& + ( \sqrt{2} \lVert \xi^{-2} \hat{u}_{0}( \xi ) \rVert_{L^{2}( \lvert \xi \rvert < 1 )} + C_{2} ) \\
	& + ( \sqrt{2} \lVert \xi^{-1} ( \partial_{\xi} \hat{u}_{0} )( \xi ) \rVert_{L^{2}( \lvert \xi \rvert < 1 )} + C_{3} )
	\}
	\end{array}
\end{equation}
when $\xi^{-2} \hat{u}_{0}( \xi ) \in L^{2}( \lvert \xi \rvert < 1 )$ and $\xi^{-1} ( \partial_{\xi} \hat{u}_{0} )( \xi ) \in L^{2}( \lvert \xi \rvert < 1 )$.
Hence,
\begin{equation}
	\lvert u( t, x ) \rvert
	\leq 
	C ( 1 + \lvert x \rvert ) \lvert t \rvert^{-1},
\end{equation}
where $C > 0$ is a constant which depends only on the behavior of $\hat{u}$ near $\xi = 0$.
This statement completes the proof of the theorem.

\end{document}